\documentclass[12pt]{tac}

\usepackage[square]{natbib} 
\usepackage{dutchcal}
\usepackage{amssymb}
\usepackage{amsmath}
\usepackage{tikz-cd}

\newcommand{\sset}{\mathrm{sSet}}
\newcommand{\stinfty}{\omega \mathrm{Gpd}}
\newcommand{\algkan}{\mathrm{AlgKan}}
\newcommand{\stcom}{\mathrm{sTCom}}
\newcommand{\crcom}{\mathrm{CrCom}}
\newcommand{\ch}{\text{Ch}_\mathbb{Z}^+}
\newcommand{\sabgroup}{\mathrm{sAbGrp}}
\newcommand{\qsset}{\mathcal{sSet}}
\newcommand{\qstcom}{\mathcal{sTCom}}

\newcommand{\C}{\mathcal{C}}
\newcommand{\D}{\mathcal{D}}
\newcommand{\lf}{\mathcal{L}}
\newcommand{\rf}{\mathcal{R}}
\mathrmdef{id}
\mathrmdef{holim}
\mathrmdef{fill}
\mathrmdef{Hom}
\mathrmdef{Aut}
\renewcommand{\phi}{\varphi}

\newcommand{\stalg}{\text{St}_{\text{Alg}}}
\newcommand{\ualg}{\text{U}_{\text{Alg}}}
\newcommand{\st}{\text{St}}
\newcommand{\ust}{\text{U}_{\st}}

\newcommand{\abtcom}{\text{Ab}_{\text{sT}}}
\newcommand{\utcom}{\text{U}_{\text{sT}}}
\newcommand{\abcomplex}{\text{Ab}_{\text{Cr}}}
\newcommand{\ucomplex}{\text{U}_{\text{Cr}}}
\newcommand{\doldnormalizer}{\text{N}_\mathbb{Z}}
\newcommand{\doldsset}{\Gamma_\mathbb{Z}}
\newcommand{\nonabdoldnormalizer}{\text{N}_{\crcom}}
\newcommand{\nonabdoldsset}{\Gamma_{\crcom}}
\newcommand{\tz}{\tilde{\mathbb{Z}}}
\newcommand{\tabcomplex}{\tilde{\text{Ab}}_{\text{Cr}}}
\newcommand{\freecoalg}{\mathcal{F}_{\st}}
\begin{document}
%-------------------------------------------------------------------------------
\author{Kimball Strong}
\address{Department of Mathematics, Cornell University}
\eaddress{ks2424@cornell.edu}
\thanks{
The author would like to thank his advisor, Inna Zakharevich. 
He would also like to thank Charles Rezk for his insights and encouragement, 
and for originally posing the question answered by the title of this paper.
The author was partially supported by NSF grant DMS-2052977 during the writing of this paper.}
%-------------------------------------------------------------------------------
\title{Strictification of $\infty$-groupoids is comonadic}
\maketitle
\begin{abstract}
	We investigate the universal strictification adjunction from 
	weak infininity-groupoids (modeled as simplicial sets) to 
	``strict infinity-groupoids,'' more commonly called ``omega-groupoids.'' 
	Modeling these with simplicial T-complexes, 
	we prove that any simplicial set can be recovered 
	up to weak homotopy equivalence as the totalization 
	of its canonical cosimplicial resolution induced 
	by this adjunction. 
	We explain how this generalizes the fact due to Bousfield and Kan 
	that the homotopy type of a simply connected space can be recovered 
	as the totalization of its canonical cosimplicial resolution induced 
	by the free simplicial abelian group adjunction. 
	Furthermore, we leverage this result to show that 
	this strictification adjunction induces a comonadic adjunction 
	between the quasicategories of simplicial sets and omega-groupoids.
\end{abstract}

\tableofcontents
\section{Introduction}	
A basic goal of classical algebraic topology is to construct 
algebraic invariants of homotopy types: 
functors $F: \text{Ho}(\text{Top}) \to \mathcal{A}$ 
where $\text{Ho}(\text{Top})$ is the category of topological spaces 
localized at the weak equivalences and 
$\mathcal{A}$ is some category of algebraic objects, e.g. groups or rings. 
Grothendieck's \textit{Homotopy Hypothesis} is the statement that 
there is a \textit{complete} such invariant generalizing the notion of groupoid, 
which he called an $\infty$-groupoid. 
By ``complete'' we mean that there is a notion of ``weak equivalence'' 
for $\infty$-groupoids, and that there is an equivalence of categories
$$ \text{Ho}(\text{Top}) \simeq \text{Ho}(\infty\text{-Groupoid}) $$
The idea of an $\infty$-groupoid is based on axiomatizing 
the algebraic structure that the points, paths, homotopies, 
homotopies between homotopies, etc. in a topological space carry. 
Unfortunately, this information is too unwieldy to work with directly---
Grothendieck provided a complete definition, 
but no one has yet been able to prove or disprove 
that it leads to the desired equivalence of categories. 
Various alternate definitions along the same lines have been proposed, 
see for instance \citep{Cisinski_2006, Henry_2016}.
\\
\indent The enormous success of simplicial methods in homotopy theory led to a solution of a different sort:
taking ``$\infty$-groupoid'' to mean ``Kan complex.'' 
Kan complexes certainly satisfy the condition of modeling the homotopy theory of spaces.
If one is interested not necessarily in understanding \
Grothendieck's algebraic vision of ``$\infty$-groupoids'' 
and is mainly looking for some convenient object with which to do homotopy theory, 
there is no reason to seek anything else. 
However, Kan complexes do not entirely fit the bill 
of axiomatizing the algebraic data in a space: 
they are not particularly algebraic 
(in the sense of being equipped with operations satisfying certain relations). 
One can define ``algebraic Kan complexes,'' 
in which horn filling is an operation (rather than a property); 
in \cite{Nikolaus2011} it is shown that these model spaces. 
This still falls somewhat short of the original vision---
Kan complexes effectively hide much of the complicated nature of $\infty$-groupoids 
by relegating it to simplicial combinatorics. 
For instance: $S^2$ is the ``free $\infty$-groupoid on a single $2$-cell.'' 
Interpreting ``free $\infty$-groupoid on a $2$-cell'' in Kan complexes 
gives a model for the $2$-sphere, and while in principle you can compute 
anything you like from this, nothing is  ``intrinsically obvious.'' 
By contrast, in the weak $3$-groupoid model 
(which one can use to model the homotopy $3$-type of $S^2$), 
you can obtain that $\pi_3(S^2) \cong \mathbb{Z}$ directly from the axioms: 
the generator comes from a coherence $3$-cell which expresses 
homotopy-commutativity of composition of $2$-cells 
(this is an axiomatization of the Eckmann--Hilton argument). 
Essentially, more homotopical information is encoded 
``by hand'' into the axioms of a weak $3$-groupoid than in Kan complexes. 
\\
\indent Unfortunately, encoding homotopical information directly into axioms 
for weak $n$-groupoids becomes infeasible as $n$ gets larger. 
Fortunately, one can still retain a significant amount of information 
by working with ``\textit{strict}'' $\infty$-groupoids, 
which we for clarity shall refer to exclusively as $\omega$-groupoids, 
which are much easier to define and manipulate:
\begin{definition}
An \textbf{$\omega$-groupoid} is a sequence of sets
\begin{center}
\begin{tikzcd}[sep = huge]
	X_0 
		\ar[r, "{\id}"] 
	& X_1 
		\ar[r, "{\id}"] 
		\ar[l, "s", shift left = 2,bend left = 15] 
		\ar[l,"t" swap, shift right = 2, bend right = 15] 
	& X_2 
		\ar[r, "{\id}"] 
		\ar[l, "s", shift left = 2, bend left = 15] 
		\ar[l,"t" swap, shift right = 2, bend right = 15]   
	& \cdots 
		\ar[l, "s", shift left = 2, bend left = 15] 
		\ar[l,"t" swap, shift right = 2, bend right = 15] 
\end{tikzcd}
\end{center}
such that each diagram
\begin{center}
\begin{tikzcd}[sep = huge] 
	X_i 
		\ar[r, "{\id^k}"]
	& X_{i+k} 
		\ar[l, "s^k", shift left = 2, bend left = 15] 
		\ar[l,"t^k" swap, shift right = 2, bend right = 15] 
\end{tikzcd}
\end{center}
is equipped with the structure of a groupoid, and such that these are compatible in the sense that 
\begin{center}
\begin{tikzcd}[sep = huge]
	X_i \
		\ar[r, "{\id^k}"] 
	& X_{i+k} 
		\ar[r, "{\id^j}"] 
		\ar[l, "s^k", shift left = 2, bend left = 15] 
		\ar[l,"t^k" swap, shift right = 2, bend  right = 15]  
	& X_{i+k+j} 
		\ar[l, "s^j", shift left = 2, bend left = 15] 
		\ar[l,"t^j" swap, shift right = 2, bend  right = 15] 
\end{tikzcd}
\end{center}
is a strict $2$-groupoid. 
A map between $\omega$-groupoids is a map of diagrams which 
preserves all the groupoidal structure. 
The resulting category we notate as $\stinfty$.
\end{definition}
One can define a functor $\text{Strict}: \sset \to \stinfty$, 
but it is well known that $\omega$-groupoids are not a complete invariant for homotopy types 
(see \citep{Ara_2013} for an overview of how much homotopical information they can model; 
essentially it is a mixture of the fundamental group and higher homological information). 
They nonetheless provide a useful concept while trying to understand 
the general problem of constructing algebraic invariants, 
in particular the problem of giving a convenient definition of $\infty$-groupoids. 
We think of the functor $\text{Strict}: \sset \to \stinfty$ as giving the 
``strictification'' of an $\infty$-groupoid 
(presented up to homotopy equivalence by a simplicial set). 
It has a right adjoint $U: \stinfty \to \sset$.
\\
\indent The main goal of this paper is to examine this strictification functor, 
and in particular to prove that it induces a comonadic adjunction of quasicategories. 
One can think of these results as providing a complete algebraic model for homotopy types---
coalgebras in $\stinfty$ are a working model for $\infty$-groupoid.
The weak (and somewhat imprecise) form of what we will prove is:
\begin{theorem}
	The homotopy type of a space $X$ is determined by 
	the homotopy type of its strictification $\text{Strict}(X)$, 
	along with a natural coalgebra structure over the comonad 
	induced by the adjunction $\text{Strict} \dashv U$.
\end{theorem}
The stronger form is 
\begin{theorem}
	The strictification functor is comonadic on the level of quasicategories: 
	that is, it induces an equivalence of quasicategories between spaces 
	and the quasicategory of coalgebras for the comonad $\text{Strict} \circ U$ on $\stinfty$.
\end{theorem}
These appear more precisely as Theorems 
\ref{thm:main-theorem-basic} and \ref{thm:main}, respectively.
\\
\indent In relation to previous work, this is a generalization of 
the main result on homotopy categories in \citep{Blomquist_Harper_2019}, 
in which the authors prove an analagous result 
for simply connected spaces and chain complexes. 
As simply connected chain complexes are equivalent to 
simply connected strict $\infty$-groupoids, 
restricting our result to the simply connected case recovers their main result.\footnote{
	One advantage of the approach by the authors there 
	is that they work directly with a simplicially enriched category of coalgebras, 
	whereas we utilize the Barr-Beck-Lurie theorem, 
	which lets us avoid explicitly constructing a category of coalgebras.
}
Our Theorem \ref{thm:main-theorem-basic} is a direct generalization 
of the fact that the $\mathbb{Z}$-completion 
(in the sense of Bousfield-Kan) of a nilpotent space 
is itself to the non-nilpotent case. 
Another closely related result is in \citep{Rivera_Wierstra_Zeinalian_2022}, 
where the authors succeed in giving an equivalence of homotopy theories between 
localizations of spaces at algebraically closed fields 
and simplicial coalgebras (with no simple connectedness assumptions). 
\\
\indent In order to arrive at our result we will actually do 
very little directly with $\omega$-groupoids, 
or with topological spaces. Instead, we will work with simplicial models: 
instead of topological spaces we will work with simplicial sets, 
and instead of $\omega$-groupoids we will work with simplicial $T$-complexes, 
which are a certain sort of Kan complex which form a category 
equivalent to the category of $\omega$-groupoids.
\section{Technical background}
\subsection{Algebraic Kan complexes}
The category of simplicial sets models the homotopy theory of $\infty$-groupoids 
under the model structure in which the fibrant-cofibrant objects are Kan complexes 
and weak equivalences are those maps which induce isomorphisms 
on all homotopy groups, for any choice of basepoint.
As the composition operation in a Kan complex is a relation rather than a function 
(that is, compositions are asserted to exist rather than given as the output of a composition function), 
it is difficult to work directly with Kan complexes to provide 
a ``strictification'' of their $\infty$-groupoid structure. 
Therefore, we work with the notion of an 
\textbf{algebraic Kan complex,} due to \citep{Nikolaus2011}.
\begin{definition}[\citep{Nikolaus2011}, definition 3.1]
	The category $\text{AlgKan}$ is defined as follows: 
\begin{itemize}
	\item The objects are simplicial sets $X$ such that for every diagram 
\begin{center}
	\begin{tikzcd}
	\Lambda^n_k \ar[d, "\iota^n_k"] \ar[r, "h"] & X \\
	\Delta^n 
	\end{tikzcd}
\end{center}

there is a \emph{chosen} horn filler $\fill_X(h): \Delta^n \to X$ 
which makes the diagram commute. 
In other words, objects are pairs $(X, \fill_X)$ 
where $\fill_X$ is a function from horns of $X$ to horn fillers. 
For a horn $h: \Lambda^n_k \to X$, 
we call $d_k(\fill_X(h))$ the \emph{composition} of the horn $h$.
We call the simplices in the image of $\fill_X$ \emph{distinguished fillers}.
\item The morphisms are morphisms of underlying simplicial sets 
which preserve the fillings: 
a morphism $f : (X, \fill_X) \to (Y, \fill_Y)$ 
is a map of simplicial sets $f : X \to Y$ 
such that for any horn $h: \Lambda^n_k \to X$, 
$$f(\fill_X(h)) = \fill_Y(f \circ h)$$
\end{itemize}	
\end{definition}
\begin{example}
	Denote by $|-|$ and $S_\bullet(-)$ the functors for the 
	geometric realization of a simplicial set and the singular simplicial set 
	of a topological space, respectively. 
It is straightforward to functorially equip the singular simplicial set 
of a topological space with choices of horn fillers: 
for each of the horn inclusions $\Lambda^n_k \hookrightarrow \Delta^n$, 
fix a retract $R^n_k: |\Delta^n| \twoheadrightarrow |\Lambda^n_k|$. 
Then for any topological space $X$ and any diagram 
\begin{center}
\begin{tikzcd}
	\Lambda^n_k 
		\ar[d, "\iota^n_k"] 
		\ar[r, "f"] 
	& S_\bullet(X) \\
	\Delta^n 
\end{tikzcd}
\end{center}
We have an adjoint diagram 
\begin{center}
\begin{tikzcd}
	{| \Lambda^n_k |} 
		\ar[d, "|\iota^n_k|"] 
		\ar[r,"|f|"] 
	& X \\
	{| \Delta^n |}
\end{tikzcd}
\end{center}
And the adjoint of the map $|f| \circ R^n_k$ gives us the horn filler. 
This equips each horn with a choice of horn filler, 
defining a functor $AlgS_\bullet$, such that we have a factorization 
\begin{center}
\begin{tikzcd}
	 \text{Top} 
	 	\ar[rd, "AlgS_\bullet"] 
	 	\ar[rr, "S_\bullet"] 
	 & & \text{sSet} \\
	 & \text{AlgKan} 
	 	\ar[ru, "U_A"]
\end{tikzcd}
\end{center}
Where $U_{A}$ is the evident forgetful functor. 
In \citep{Nikolaus2011}, it is shown that the functor 
$U_A$ is the right adjoint of a Quillen equivalence, 
and thus the category $\text{AlgKan}$ models 
the homotopy theory of $\infty$-groupoids. 
\end{example}
\subsection{Simplicial T-complexes}
Suppose that we have an algebraic Kan complex $(X, \fill_X)$, 
and two composable simplices $f,g \in X_1$. Then the diagram 
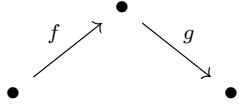
\begin{center}
\begin{tikzcd}
	& \bullet 
		\ar[rd, "g"] \\
	\bullet 
		\ar[ru, "f"] 
	& & \bullet 
\end{tikzcd}
\end{center}
Can be canonically filled to obtain the composition $gf$. We can then consider the diagram
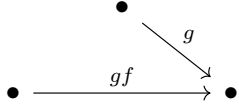
\begin{center}
\begin{tikzcd}
	& \bullet 
		\ar[rd, "g"] \\
	\bullet \ar[rr, "gf"] 
	& & \bullet 
\end{tikzcd}
\end{center}
And again use our choice of horn filler 
to obtain a composition $g^{-1}(gf)$. 
Now, if we working in a strict groupoid, 
we would have that this is equal to $f$. 
However, in the context of a weak $\infty$-groupoid 
we can conclude that $g^{-1}(gf)$ is homotopic to $f$, 
but not necessarily equal to it. 
If we wish to model strict groupoids, therefore, 
we must impose some sort of compatibility conditions on our fillers.
\\
\indent Suppose that we have a horn $h: \Lambda^n_k \to X$, 
and its composition filler is $\fill_X(h): \Delta^n \to X$. 
Then, if we are trying to model strict higher groupoids, 
a reasonable thing to assert is that $\fill_X(h)$ 
is the composition filler all of its horns, not just the $k$th.  
In the above example, this would mean that since we have a filler 
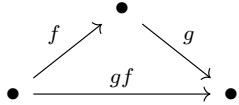
\begin{center}
\begin{tikzcd}
	& \bullet 
		\ar[rd, "g"] \\
	\bullet 
		\ar[rr, "gf" ] 
		\ar[ru, "f"] 
	& & \bullet 
\end{tikzcd}
\end{center}
The composition filler for $g^{-1}(gf)$ must be the again 
the same $2$-simplex, so that $f = g^{-1}(gf)$. 
This motivates the following definition, 
which appeared originally in \citep{Dakin_1977}. 
\begin{definition}[\citep{Dakin_1977}, Definition 1.1]
	A \textbf{Simplicial T-Complex} is a simplicial set $X$ 
	equipped with a set of marked simplices, 
	which we refer to as ``thin.'' 
	These are not required to form a subcomplex, 
	but must satisfy the following axioms:
	\begin{enumerate}
		\item Every degenerate simplex is thin. 
		\item For each horn $h: \Lambda^n_k \to X$, 
		there is a \emph{unique} filler $C: \Delta^n \to X$ such that
		the image of the nondegenerate $n$-simplex is thin in $X$. 
		Given such an extension, we refer to the $k$th boundary 
		of this filler as the \emph{composition} of the horn $h$. 
		\item If $h: \Lambda^n_k \to X$ is a horn with all 
		nondegenerate $(n-1)$ simplices thin, 
		then the composition of $h$ is thin as well.
	\end{enumerate}
	A map of simplicial $T$-complexes is a map on underlying simplicial sets 
	which sends thin simplices to thin simplices. 
	We denote the resulting category by $\stcom$.
\end{definition} 
Condition $1$ imposes that the composition of a simplex 
with degenerate simplices is the original simplex. 
Condition $2$ is the compatibility condition we reasoned out above. 
Condition $3$ is slightly more subtle, and is best illuminated by example: 
suppose we have a diagram of three composable $1$-simplices:
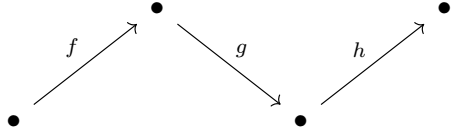
\begin{center}
\begin{tikzcd}[sep = large]
	& \bullet 
		\ar[rd, "g"] 
	& & \bullet \\
	\bullet 
		\ar[ru, "f" ] 
	& & \bullet 
		\ar[ru, "h" ]
\end{tikzcd}
\end{center}
Then we can fill in the horn defined by $f$ and $g$, 
and the horn defined by $g$ and $h$, giving us:
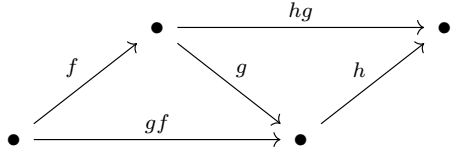
\begin{center}
\begin{tikzcd}[sep = large]
	& \bullet \ar[rd, "g"] 
		\ar[rr, "hg"] 
	& & \bullet \\
	\bullet \ar[ru, "f" ] 
		\ar[rr, "gf"] 
	& & \bullet \ar[ru, "h" ]
\end{tikzcd}
\end{center}
Then we can furthermore fill in the horn defined by $f$ and $hg$, giving us
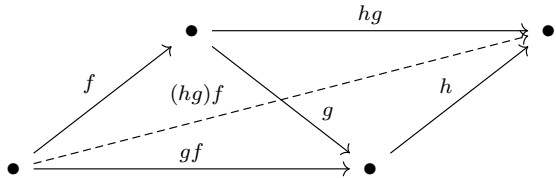
\begin{center}
\begin{tikzcd}[sep = huge]
	& \bullet \ar[rd, "g", near end]
		\ar[rr, "hg"] & & \bullet \\
	\bullet \ar[ru, "f" ] 
		\ar[rr, "gf"] 
		\ar[rrru, dashed, "(hg)f" {pos=0.4} ] 
	& & \bullet \ar[ru, "h" ]
\end{tikzcd}
\end{center}
At this point, we have a horn $\Lambda^3_1$ where every face is thin. 
By axiom $3$, the filler is thin, and therefore, looking at face $1$, 
which we get by filling, we have $(hg)f = h(gf)$. 
So condition $3$ is a uniqueness condition that ensures 
a composition is uniquely defined regardless of 
the order we choose to do compositions (that is, fill horns) in. 
Concretely, if we have already obtained 
every $n-1$-simplex of an $n$-horn via composition of the $n-2$ simplices, 
then the additional $n-1$ simplex we get should also express composition. 
\\
\indent Another simpler motivation is this: suppose that we have a map 
$\partial \Delta^n \to X$ where each face is thin. 
Then we may look at any of the sub-horns, 
$\Lambda^n_k \hookrightarrow \partial \Delta^n \hookrightarrow X$, 
and obtain a composition of this sub-horn, giving us a thin filler 
$\Lambda^n_k \hookrightarrow \Delta^{n}  \to X$. 
Imposing condition (2) tells us that the composition must be thin, 
and by uniqueness of thin fillers, we must have that 
it is equal to the $k$th face of our original map 
$\partial \Delta^n \to X$. 
This implies that our thin filler is in fact 
a filler for every sub-horn of this map $\partial \Delta^n \to X$. 
In particular, we get that the simplicial $T$ complex formed by 
taking the $n$-skeleton of $X$ and iteratively throwing in 
all thin fillers of horns is $(n+1)$-coskeletal, 
much in the way the nerve of an ordinary strict $1$-category is $2$-coskeletal. 
\\
\indent Of course, the best motivation for simplicial $T$-complexes 
is simply that they are the correct simplicial analogue 
of strict $\infty$-groupoids, which is a consequence of \citep{Ashley_78}:
\begin{theorem}
	There is an equivalence of categories $\stcom \rightleftarrows \stinfty$.
\end{theorem}
We will say more about this equivalence in Section \ref{sec:doldkan}.
\subsection{Crossed complexes}
We have already defined the category $\stcom$ of simplicial T-complexes 
and the category $\stinfty$ of $\omega$-groupoids. 
We now define a third category, the category $\crcom$ 
of crossed complexes, which are a sort of 
``nonabelian chain complex'' in a way which we will clarify shortly.
\begin{definition}[\citep{Brown_Higgins_Sivera_2011}, Definition 7.1.9]
	A \textbf{crossed complex} $C$ is a sequence of sets 
	\begin{center}
	\begin{tikzcd}
		C_0 
			\ar[r, "{\id}"] 
		& C_1 
			\ar[l, "s", shift right = 2, bend right = 15, swap] 
			\ar[l, "t", shift left = 2, bend left = 15] 
		& C_2 
			\ar[l, "\delta_2"] 
		&  C_3 
			\ar[l,"\delta_3"] 
		& \cdots 
			\ar[l]
	\end{tikzcd}
	\end{center}
	Such that:
	\begin{enumerate}
		\item The diagram
		\begin{center}
		\begin{tikzcd}
		C_0 
			\ar[r, "{\id}"]
		& C_1 
			\ar[l, "s", shift right = 2, bend right = 15, swap] 
			\ar[l, "t", shift left = 2, bend left = 15]
		\end{tikzcd}
		\end{center}
		forms a groupoid, 
		which we will abuse notation by referring to simply as $C_1$.
		\item Each $C_i$ for $i \ge 2$ is a skeletal module 
		over the groupoid $C_1$: that is, a family of groups of the form 
		$$C_i = \coprod_{c \in C_0} C_i(c)$$
		where each $C_i(c)$ is a group, equipped with morphisms
		$$
		\phi_\ell: C_i(s(\ell)) \to C_i(t(\ell))
		$$
		for each $\ell \in C_1$, 
		satisfying that for composable $\ell$ and $p$ in $C_1$, 
		$$
		\phi_{\ell \circ p} = \phi_\ell \circ \phi_p
		$$
		and that $\phi_{\id_x} = \id_{C_i(x)}$. 
		Further, each $C_i(c)$ is abelian for $i>2$. 
		From now on we shall generally suppress the 
		$ c \in C_0$ from our notation when our meaning is clear, 
		saying for example ``$C_i$ is abelian for $i > 2$.''
		\item For $i > 2$, the maps $\delta_i$ are families of maps of groups 
		$\delta_i : C_i \to C_{i-1}$, 
		satisfying $\delta_{i-1} \circ \delta_i = 0$.
		\item $\delta_2$ is a family of maps of groups 
		$\delta_2(c) : C_2(c) \to \Aut(c)$, 
		where by $\Aut(c)$ we mean the automorphism group of $c$ in the groupoid $C_1$.
		\item The action of $C_1$ on $C_i$ is compatible with the 
		$\delta_i$ in the sense that for $i > 2$, 
		$\ell \in C_1$, and $a \in C_i(x)$
		$$\phi_\ell \circ \delta_i = \delta_{i-1} \circ \phi_\ell$$
		\item For any $a \in C_2$, $\delta_2(a)$ acts 
		by conjugation by $a$ on $C_2$ and trivially on $C_i$ for $i > 2$.
	\end{enumerate}
\end{definition}
To a crossed complex $C$, we can associate homology groups:
\begin{definition}
	For $C \in \crcom$ and $c \in C_0$, 
	define $\pi_0$, $\pi_1$, and $H_n$ for $n \ge 2$ by the following:
	\begin{itemize}
		\item $\pi_0(C)$ is $\pi_0$ of the groupoid $C_1$. 
		\item $\pi_1(C,c)$ is $\text{Aut}(C_1,c)/\delta_2(C_2(c))$
		\item For $n \ge 2$, $H_n(C,c)$ is $\text{ker}(\delta_n)(c)/\text{im}(\delta_{n+1})(c)$.
	\end{itemize}	 
	Note that $H_2$ is always abelian although $C_2$ may not be, 
	as the condition that $\delta(C_2)$ acts on itself by conjugation 
	means that $\text{ker}(\delta_2)$ commutes with everything in $C_2$.
\end{definition}
Our interest in $\crcom$ is motivated by the following theorem, 
which is the primary content of \citep{Brown_1981}:
\begin{theorem}\label{crossed-strict-equivalence}
	There is an equivalence of categories $\crcom \rightleftarrows \stinfty$.
\end{theorem}
Thus, both $\stcom$ and $\crcom$ can be used to analyze $\stinfty$, 
and we will use both towards proving our main theorem.
\section{The Dold-Kan correspondences: abelian and nonabelian}\label{sec:doldkan}
We use $\sabgroup$ to denote the category of simplicial abelian groups 
and $\ch$ to denote the category of nonnegatively graded 
chain complexes of abelian groups. 
In this section, we will recall the adjoint pairs in the following diagram:
\begin{center}
\begin{equation}\label{eqn:functors-square}
\begin{tikzcd}[sep = huge]
		 \stcom 
		 	\ar[rr, bend left = 20, "\abtcom"]  
		 	\ar[dd, bend left = 20, "\nonabdoldnormalizer"]
		 & &  \sabgroup 
		 	\ar[ll, bend left = 20, "\utcom"] 
		 	\ar[dd, bend left = 20, "\doldnormalizer"] \\ \\
		\crcom 
			\ar[rr, bend left = 20, "\abcomplex"] 
			\ar[uu, bend left = 20, "\nonabdoldsset"]
		& &  \ch 
			\ar[ll, bend left = 20, "\ucomplex"] 
			\ar[uu, bend left = 20, "\doldsset"]
\end{tikzcd}
\end{equation}
\end{center}
Here the vertical edges are equivalences of categories, 
the left facing arrows are left adjoints, 
and the right facing arrows are the corresponding right adjoints. 
The right edge is the best known: it is the Dold-Kan correspondence. 
The left edge is a sort of nonabelian version, due to \citep{Ashley_78}. 
The functors on the top and bottom edges we believe are known, 
with the right adjoint appearing in \citep{Brown_Higgins_Sivera_2011},\footnote{
	Where it is called $\Theta$ and defined more generally 
	for chain complexes with a groupoid of operators.}
but lacking a comprehensive reference we fully define them here. 
The top and bottom edges each form adjunctions, 
but not equivalences of categories or homotopy theories. 
However, they induce equivalences of homotopy theories of $1$-connected objects, 
a fact which will be key to us. 
The main result of this section which we will use in proving 
our main theorem is the following.
\begin{theorem}\label{thm:simply-connected-crossed-complex}
	There are adjoint pairs of functors as in 
	Diagram \ref{eqn:functors-square} above satisfying: 
	\begin{enumerate}
		\item The left and right edges are equivalences.
		\item For $A \in \sabgroup$, 
		the underlying simplicial set of $\utcom(A)$ 
		is the same as the underlying simplicial set of $A$.
		\item For $C \in \crcom$ with $C$ $1$-reduced 
		(that is, $C_0$ and $C_1$ are singletons), 
		we have that the natural map $C \to \ucomplex(\tilde{\abcomplex}(C))$ 
		induces an isomorphism, where $\tilde{\abcomplex}(C)$ 
		is the reduced chain complex coming from $\abcomplex(C)$.
		\item The diagram is commutative in the sense that there is a natural isomorphism 
		$\utcom \circ \doldsset \cong \nonabdoldsset \circ \ucomplex$.
	\end{enumerate}
\end{theorem}	 
The proof is spread throughout this section.
\subsection{The right edge: the Dold-Kan correspondence}
For $A \in \sabgroup$, denote by $D(A_n)$ the subgroup of $A_n$ generated by the degenerate simplices.
\begin{definition}[The functor $\doldnormalizer$]
	The functor $\doldnormalizer: \sabgroup \to \ch$ is defined by 
	$\doldnormalizer(A)_n = A_n/D(A_n)$, 
	with differential the alternating sum of the face maps.
\end{definition}	
The Dold-Kan correspondence is the following well-known result:
\begin{theorem}[The Dold-Kan Correspondence]
	There is an equivalence of categories 
	$\doldnormalizer: \sabgroup \rightleftarrows \ch : \doldsset $, 
	where for $C_\bullet \in \ch$, 
	$$
		\doldsset(C_\bullet)_n := \bigoplus_{[n] \twoheadrightarrow [k]} C_k		
	$$
\end{theorem}
For a reference, see \citep{Matthew}. 
We will need the following fact about this adjunction:
\begin{lemma}\label{lem:simplices-in-doldsset}
	For $A_\bullet \in \ch$, the $n$-simplices of $\doldsset(A_\bullet)$ 
	are in bijection with chain complex maps 
	$\doldnormalizer(\mathbb{Z}[\Delta^n]) \to A_\bullet$. 
	Under this identification, the subgroup generated by the 
	degenerate simplices consists of those maps which send 
	the unique nondegenerate $n$-simplex to $0$. 
\end{lemma}
\begin{proof}
	The bijection follows from a general fact about adjunctions on 
	$\sset$ being defined by cosimplicial objects. 
	The second fact is less trivial; see the above reference for details. 
\end{proof}
\subsection{The left edge: the nonabelian Dold-Kan correspondence}
In \citep{Ashley_78}, the following is proved: 
\begin{theorem}\label{thm:simplicial-crossed-equivalence}
	There is an equivalence of categories 
	$\nonabdoldnormalizer :  \crcom \rightleftarrows \stcom : \nonabdoldsset$.
\end{theorem}
\begin{proof}
	For a full proof, see \citep{Ashley_78}. 
	We describe here just the functor $\nonabdoldnormalizer: \crcom \to \stcom$, 
	as this is the only particular we will make use of: 
	for $C \in \crcom$, the $n$-simplices of $\nonabdoldnormalizer(C)$ are defined inductively as follows:
	\begin{itemize}
		\item For $n = 0$, they are simply $C_0$.
		\item For $n = 1$, they are simply $C_1$, 
			with boundary maps $d_0$ and $d_1$ defined as 
			the source and target maps of the groupoid $C_1$. 
		\item For $n \ge 2$ they are tuples $x = (x_0,...,x_n;\alpha)$ 
		where $x_i \in \nonabdoldnormalizer(C)_{n-1}$, $d_i(x_j) = d_{j-1}(x_i)$ 
		for $i < j$, and $\alpha \in C_n(d_1\cdots d_{n-1}(x_n))$. 
		Further, they must satisfy
		$$
			\delta(\alpha) = 
			\begin{cases}
				x_2x_0x_1^{-1} & \text{if }n=2 \\
				\phi_{p^{-1}}(x_0)x_2x_1^{-1}x_3^{-1} & \text{if } n = 3 \\
				\phi_{p^{-1}}(x_0) \sum_{i = 1}^n (-1)^i x_i & \text{if } n \ge 4
			\end{cases}		
		$$
		Here $p = d_2 d_3\cdots d_{n-1}(x_n)$, 
		and $\phi$ is the action of $C_1$ on $C_n$ for $n \ge 2$. 
		The boundary maps are defined as $d_i(x) = x_i$. 
		The degeneracy maps are defined inductively as follows: for $x = (x_0,\dots,x_n;\alpha)$
		$$s_i(x) = (s_{i-1}x_0,...,s_{i-1}x_{i-1},x,x,s_{i+1}x_i,...,s_i x_{n};e)$$
		Where $e$ is the identity element of (the appropriate component of) $C_{n+1}$.
	\end{itemize}
	For a simplex $x = (x_0,...,x_n;\alpha) \in \nonabdoldnormalizer(C)$, 
	we will call $\alpha$ as the \textit{automorphism element} of the simplex, 
	and refer to it by $\text{aut}(x)$. 
	One can think of the simplices of $\nonabdoldnormalizer(C)$ as being 
	inductively built out of compatible automorphism elements. 
	A simplex is thin precisely when $\alpha$ is the identity element.
\end{proof}
We have now finished verifying item (1) of 
Theorem \ref{thm:simply-connected-crossed-complex}. 
This equivalence of categories respects the natural 
homotopical structures on the categories, 
in particular the fundamental algebraic invariants:
\begin{lemma}\label{lem:weak-equiv-t-complex}
	For $C \in \crcom$ and $x \in C_0$, there are natural isomorphisms 
	$$\pi_0(C) \cong \pi_0(\nonabdoldnormalizer(C)) 
	\quad \pi_1(C,x) \cong \pi_1(\nonabdoldnormalizer(C),x) 
	\quad H_n(C,x) \cong \pi_n(\nonabdoldnormalizer(C),x)$$
	For $n \ge 2$.
\end{lemma}
\begin{proof}
	In \citep{Brown_Higgins_1991} they prove the analogous statement 
	for the the classifying space of a crossed complex. 
	Since the classifying space of a crossed complex $C$ 
	is defined to be the geometric realization of $\nonabdoldnormalizer(C)$, 
	and geometric realization preserves homotopy groups, the result follows. 
\end{proof}
\subsection{The top edge: the adjunction $\stcom \rightleftarrows \sabgroup$}
We first define the inclusion functor $\utcom : \sabgroup \to \stcom$.  
\begin{definition}[The functor $\utcom$]
	For $A$ a simplicial abelian group, 
	$\utcom(A)$ is the simplicial $T$-complex whose 
	underlying simplicial set is the underlying simplicial set of $A$, 
	and whose thin simplices are the sums of degenerate simplices. 
	In \citep{Ashley_78}, it is proven that 
	this gives a simplicial $T$-complex structure. 
	As maps of simplicial abelian groups preserve both sums and degenerate simplices, 
	this is functorial.
\end{definition}
We now define the left adjoint $\abtcom: \stcom \to \sabgroup$. 
Recall $\mathbb{Z}[-]$, the free simplicial abelian group functor, 
given by applying the free abelian group functor to every level of a simplicial set.
\begin{definition}[The functor $\abtcom$]
	Let $X \in \stcom$. 
	Let $|X|$ be its underlying simplicial set, 
	and $\eta$ the natural map of simplicial sets $|X| \to \mathbb{Z}[|X|]$. 
	We define
	$$\abtcom(X) = \frac{\mathbb{Z}[|X|]}
	{\langle \fill(\eta\circ h) - \eta \circ \fill(h) | h : \Lambda^n_k \to |X|\rangle}$$
	Where $\fill(h)$ denotes the 
	unique thin filler of a horn in a simplicial $T$-complex.
\end{definition}	
\begin{theorem}
	$\abtcom$ is the left adjoint to $\utcom$.
\end{theorem}
\begin{proof}
	Let $f: X \to \utcom(A)$ be a map of simplicial $T$-complexes. 
	This induces a map $\mathbb{Z}[|X|] \to A$, 
	which is a map of simplicial abelian groups 
	and therefore a map of simplicial $T$-complexes. 
	In particular, this map preserves thin fillers. 
	It follows by the universal property of the quotient 
	that this induces a map $\abtcom(X) \to A$ of simplicial abelian groups.
	\\
	\indent Conversely, let $g: \abtcom(X) \to A$ 
	be a map of simplicial abelian groups. 
	Then we get a composite map of simplicial $T$-complexes 
	$X \to \utcom(\abtcom(X)) \to \utcom(A)$. 
\end{proof}
We have now verified item (2) of Theorem \ref{thm:simply-connected-crossed-complex}.
\subsection{The bottom edge: the adjunction $\crcom \rightleftarrows \ch$}
The adjunction between crossed complexes and chain complexes will involve the most detail:
\begin{definition}[The functor $\ucomplex$]
	The right adjoint $\ucomplex: \ch \to \crcom$ is given by 
	taking a chain complex $A_\bullet$ to the crossed complex
	\begin{center}
		\begin{tikzcd}[sep = large]
			{|A_0|} 
				 \ar[r, "{\id}\times{\mathbf{0}}"] 
			& {|A_0|} \times {|A_1|} 
				\ar[l,"\pi_0", shift left = 2, bend left = 15] 
				\ar[l, "\pi_0 + d", shift right = 2, bend right = 15, swap] 
			& \coprod_{|A_0|} A_2 
				\ar[l] 
			& \amalg_{|A_0|} A_3 
				\ar[l] 
			& \cdots
		\end{tikzcd}
	\end{center}
	where $|-|$ denotes the underlying set of an abelian group. 
	The composition operation on $|A_0| \times |A_1|$ is given by 
	$(a_1, \ell_1) \circ (a_2, \ell_2) = (a_1, \ell_1 + \ell_2)$. 
	The action of $A_0 \times A_1$ is trivial in the sense that each $(p, \ell)$ 
	acts as the identity morphism from the $a$ component to the 
	$a + d(\ell)$ component of $\amalg_{a \in |A_0|} A_n$. 
	In particular, if $(a, \ell)$ is an automorphism, its action is trivial. 
\end{definition}
\begin{definition}[The functor $\abcomplex$]
	Let $C = $
	\begin{center}
		\begin{tikzcd}
			C_0 
			& C_1 
				\ar[l, "s", shift right = 2, swap] 
				\ar[l, "t", shift left = 2] 
			& C_2 
				\ar[l, "\delta"] 
			& C_3 
				\ar[l, "\delta" ] 
			& \cdots 
				\ar[l, "\delta"]
		\end{tikzcd}
	\end{center}
	Be a crossed complex. 
	The action of the elements $C_1$ on $C_n$ for $n\ge 2$ 
	lets us define a certain sort of quotient of this action, 
	which we write as $C_n/C_1$:
	$$C_n/C_1 := \frac{\oplus_{p \in C_0}C_n(p)}{\langle a - \phi_\ell(a) | \ell \in C_1\rangle}$$
	Then we define the chain complex $\abcomplex(C)$ to be
	\begin{center}
		\begin{tikzcd}
			\mathbb{Z}[C_0] 
			& \mathbb{Z}[C_1]/\sim 
				\ar[l, "t - s"] 
			& C_2/C_1 
				\ar[l, "\delta"] 
			& C_3/C_1 
				\ar[l, "\delta"] 
			& \cdots 
				\ar[l, "\delta"]
		\end{tikzcd}
	\end{center}
		Here the relation $\sim$ on $\mathbb{Z}[C_1]$ 
		is generated by $g \circ f \sim g + f$ for composable $g$ and $f$ in $C_1$. 
		We note that while $C_2$ is not abelian, 
		$C_2/C_1$ must be as $\delta(C_2) \subset C_1$ acts by conjugation on $C_2$.
\end{definition}
\begin{theorem}
The functors $\abcomplex$ and $\ucomplex$ as described above form an adjoint pair.
\end{theorem}
\begin{proof}
	We define a natural bijection between hom sets: 
	let $C \in \crcom$, $A_\bullet \in \ch$. 
	Given $f : \abcomplex(C) \to A_\bullet $, 
	we define $\bar{f}: C \to \ucomplex (A_\bullet)$ by:
	\begin{itemize}
		\item $\bar{f}_0 : C_0 \to |A_0|$ is the adjoint of $f_0 : \mathbb{Z}[C_0] \to A_0 $.
		\item $\bar{f}_1 : C_1 \to |A_0| \times |A_1|$ is given 
			on the first component by the composition $\bar{f}_0 \circ s$. 
			On the second component, it is the adjoint of the composition 
			$\mathbb{Z}[C_1] \to \mathbb{Z}[C_1]/\sim \to A_1$.
		\item For $n \ge 2$, recall that each $C_n$ is of the form 
			$\coprod_{p \in C_0} C_n(p)$. 
			The map $\bar{f}_n : C_n \to \coprod_{|A_0|} A_n$ is simply 
			given on each component $C_n(p)$ by the composition $C_n(p) \hookrightarrow C_n/C_1 \to A_n$.
	\end{itemize}
	Now, for the other direction: given $g : C \to \ucomplex (A_\bullet)$, 
	we define the adjoint $\bar{g}: \abcomplex(C) \to A_\bullet$ as follows:
	\begin{itemize}
		\item $\bar{g}_0: \mathbb{Z}[C_0] \to A_0$ is given by the adjoint to $g_0 : C_0 \to |A_0|$.
		\item Take the map $h: \mathbb{Z}[C_1] \to A_1$ adjoint to 
			the second component of $C_1 \to |A_0| \times |A_1|$. 
			Then for composable $\ell$ and $k$ in $C_1$, 
			we have by the definition of composition in $|A_0| \times |A_1|$ 
			that $h(\ell\circ k) = h(\ell) + h(k)$. 
			Hence, we can let $\bar{g}_1$ be the induced map $\mathbb{Z}[C_1]/\sim \to A_1$.
		\item For $n\ge 2$, we need to define a map $\bar{g}_n: C_n/C_1 \to A_n$, 
			given a map $g_n: C_n \to \coprod_{|A_0|} A_n$. 
			By definition of $C_n/C_1$, it suffices to argue that for 
			$\ell \in C_1$ and $\alpha \in C_n$, $g_n(\alpha) = g_n(\phi_\ell(\alpha))$. 
			But this follows as we said the action is trivial.
	\end{itemize}
\end{proof}
We can now conclude the proof of \ref{thm:simply-connected-crossed-complex}:
\begin{proof}[Proof of \ref{thm:simply-connected-crossed-complex}, item (3)]
	Let $C$ be a $1$-reduced crossed complex. 
	Then $\abcomplex(C)$ is the chain complex where 
	$\abcomplex(C)_0 = \mathbb{Z}$, $\abcomplex(C)_1 = 0$, $\abcomplex(C)_n = C_n$ for $n > 1$. 
	Then $\tabcomplex(C)$ is $0$ in dimensions $0$ and $1$, 
	and $\tabcomplex(C)_n = C_n$ for higher dimensions. 
	Then the map is $C \to (\ucomplex \circ \tabcomplex)(C)$ 
	is evidently the identity in dimensions $1$ and higher, 
	and in dimension $0$ is the composition 
	$\{\bullet\} \to \left|\mathbb{Z}\right| \to \left| \mathbb{Z}/\mathbb{Z} \right|$, 
	and so also an isomorphism there.
\end{proof}
\begin{proof}[Proof of \ref{thm:simply-connected-crossed-complex}, item (4)]
	We will construct a natural isomorphism 
	$\eta :\utcom \circ \doldsset\Rightarrow \nonabdoldsset \circ \ucomplex $. 
	Let $A_\bullet \in \ch$. We inductively define $\eta$ as follows:
	\begin{itemize}
		\item Since the $0$-simplices of $(\utcom \circ \doldsset)(A_\bullet)$ 
			and $(\nonabdoldsset \circ \ucomplex)(A_\bullet) $ are both $A_0$, 
			$\eta$ is the identity on $0$-simplices.
		\item The $1$-simplices of $(\utcom \circ \doldsset)(A_\bullet)$ 
			are $A_0 \oplus A_1$, with boundary maps $d_0(a_0,a_1) = a_0$ 
			and $d_1(a_0,a_1) = a_0 + d(a_1)$. 
			Since the $1$-simplices of $(\nonabdoldsset \circ \ucomplex)(A_\bullet)$ 
			are $|A_0| \times |A_1|$, we can define $\eta$ as the identity, 
			and the boundary maps agree.
		\item Let $x \in (\utcom \circ \doldsset)(A_\bullet)_n$ 
			be represented by a map $\doldnormalizer(\mathbb{Z}[\Delta^n]) \to A$, 
			and let $\alpha$ be the image of the unique nondegenerate $n$-simplex. 
			Then we let
			$$\eta(x) = (\eta(d_0(x)) , \dots , \eta(d_n(x));\alpha)$$
	\end{itemize}
	We must prove that this is a well-defined map of simplicial $T$-complexes, 
	and that it is injective and surjective. 
	This is easier to verify for the $0$ and $1$-simplices, 
	so we focus on the inductive step: 
	\begin{itemize}
		\item For $\eta$ to be well-defined we must have that 
			$d_i(\eta(d_j(x))) = d_j(\eta(d_i(x)))$ whenever 
			$d^id^j = d^jd^i$ in $\Delta$, the simplex category. 
			This follows from the inductive definition of boundaries and that 
			$d_j(d_i(x)) = d_i(d_j(x))$. Furthermore, we must have that 
			$d(\alpha) = \sum (-1)^i \text{aut}(\eta(d_i(x)))$,\footnote{
				Recall that if $x = (x_0,...,x_n;\alpha)$, then $\text{aut}(x) = \alpha$} 
			which follows from the definition of $\doldnormalizer(\mathbb{Z}[\Delta^n])$.
		\item For $\eta$ to be a map of simplicial $T$-complexes, 
			we must have that it commutes with the face and degeneracy maps, 
			and that it preserves thin simplices. 
			The former follows directly from the definitions. 
			The latter follows from the identification of 
			the thin simplices in Lemma \ref{lem:simplices-in-doldsset}.
		\item For injectivity: suppose $x,y \in (\utcom \circ \doldsset)(A_\bullet)_n$ 
			and $x \neq y$. We wish to show that $\eta(x) \neq \eta(y)$. 
			Inductively assume that $\eta$ is injective on the $n-1$ simplices. 
			Since $x \neq y$, their representing maps 
			$\doldnormalizer(\mathbb{Z}[\Delta^n]) \to A$ 
			must differ, so they must differ either in the image 
				of the nondegenerate $n$-simplex, or for some $i$ 
			we must have that $d_i(x) \neq d_i(y)$. 
			If the former, they are different as their automorphism elements are different. 
			If the latter, it follows by our inductive hypothesis.
		\item For surjectivity: let $(x_0,...,x_n; \alpha) \in (\nonabdoldsset \circ \ucomplex)(A)$. 
			Inductively assume that $\eta$ is surjective, 
			then define a map $\doldnormalizer(\mathbb{Z}[\Delta^n]) \to A$ 
			by sending the nondegenerate $n$-simplex to $\alpha$, 
			and sending the $i$th boundary to $\eta^{-1}(x_i)$.
	\end{itemize} 
\end{proof}
\section{The Quillen Adjunction $\stalg \dashv \ualg$}
\subsection{The functor $\text{AlgKan} \to \text{sTCom}$} 
We are ready to define our strictification $\stalg: \text{AlgKan} \to \stcom$. 
We will construct this as an iterated quotient of algebraic Kan complexes, 
using the work in \citep{Nikolaus2011} on the structure of the category $\text{AlgKan}$. 
\begin{definition}\label{dfn:thin-simplices}
	Let $X$ be an algebraic Kan complex. 
	We inductively call a simplex $\alpha: \Delta^n \to X$ 
	\emph{thin} if any of the following conditions hold:
	\begin{enumerate}
		\item $\alpha$ is degenerate.
		\item $\alpha$ is a distinguished filler of $X$.
		\item $\alpha$ is a composition of thin simplices.
	\end{enumerate}
\end{definition}

\begin{definition}
	Let $\alpha: \Delta^n \to X$ and $\beta: \Delta^n \to X$ 
	be thin simplices of the algebraic Kan complex $X$. 
	We call $\alpha$ and $\beta$ \textit{co-thin} if 
	there is a map $\Lambda^{n}_k \to X$ such that 
	$\alpha$ and $\beta$ make the lifting diagram
	\begin{center}
	\begin{tikzcd}[sep = large]
	\Lambda^{n}_k \ar[d] \ar[r] & X \\
	\Delta^n  \ar[ru, "\alpha{,} \beta "]
	\end{tikzcd}
	\end{center}
	commute.
\end{definition}
	In other words, two thin simplices are co-thin if they share a horn. 
\begin{lemma}\label{lem:thin-preserved}
	Let $f: X \to Y$ be a map of algebraic Kan complexes. 
	Then $f$ preserves thin simplices. 
	In particular, if $\alpha$ and $\beta$ are co-thin in $X$, 
	then $f\alpha$ and $f\beta$ are co-thin in $Y$.
\end{lemma}
\begin{proof}
	Since $f$ is a map of simplicial sets,
	if $\alpha$ and $\beta$ share a horn, 
	then $f\alpha$ and $f\beta$ will share a horn, 
	so it remains only to show that $f$ preserves thin simplices. 
	We do this inductively, checking that conditions (1)-(3) 
	in the definition of thinness are preserved by $f$:
	\begin{enumerate}
		\item if $\alpha$ is degenerate, 
			so is $f\alpha$ since $f$ is a map of simplicial sets.
		\item if $\alpha$ is a distinguished filler, 
			then so is $f\alpha$, since $f$ is a map of algebraic Kan complexes.
		\item if $\alpha$ is a composition of the thin simplices 
			$\alpha_0,...,\alpha_n$, then $f\alpha$ is a composition of 
			$f\alpha_0,...,f\alpha_n$ (because $f$ is a map of algebraic Kan complexes) 
			and each of these is thin by the inductive hypothesis.
	\end{enumerate}
\end{proof}
We can recast our definition of simplicial T-complexes in terms of thinness:
\begin{lemma}
	An algebraic Kan complex is a simplicial T-complex 
	(with thin simplices as in definition \ref{dfn:thin-simplices}) 
	if and only if it has no distinct, co-thin simplices.
\end{lemma}
\begin{proof}
	This follows directly from the definitions.
\end{proof}
\begin{lemma}
	The inclusion $\ualg: \stcom \to \algkan$ is fully faithful.
\end{lemma}
\begin{proof}
	For $X,Y \in \stcom$, a map $X \to Y$ is a map 
	of underlying simplicial sets which preserves the thin simplices. 
	By \ref{lem:thin-preserved}, all maps $\ualg(X) \to \ualg(Y)$ are of this form. 
\end{proof}
Given this setup, we construct our left adjoint as follows: 
for $X \in \algkan$, let $X_1$ be the colimit of 
the diagram with $X$ and an object $\Delta^n$ for every pair 
of co-thin maps $\alpha$ and $\beta$, with the two maps 
$\alpha,\beta : \Delta^n \to X$. 
More succinctly, let $I$ be an indexing set for all pairs $(\alpha_i, \beta_i)$ 
of co-thin simplices in $X$. 
Then $X_1$ is the coequalizer (in $\algkan$) of the diagram
\begin{center}
\begin{tikzcd}[sep = huge]
	\coprod_{i \in I}\Delta^{n_i} 
		\ar[r, bend left = 20, "\coprod_I\alpha_i"] 
		\ar[r, bend right = 20, "\coprod_I\beta_i"] 
	& X
\end{tikzcd}
\end{center}
This gives us an algebraic Kan complex $X_1$ equipped 
with a map $X \to X_1$, with the following properties:
\begin{lemma}
\label{lem:coequ-functorial}
	The assignment $X \mapsto X_1$ is functorial. 
	Let $\alpha$ and $\beta$ be co-thin maps of $X_1$. 
	If they factor through $X$, then they are equal. 
\end{lemma}
\begin{proof}
	Functoriality follows from the fact that maps in 
	$\text{AlgKan}$ must preserve distinguished fillers by 
	Lemma \ref{lem:thin-preserved}, so if we have a map $X \to Y$, 
	this induces maps between the diagrams which define $X_1$ and $Y_1$.
	The second claim follows from the definition of the colimit. 
\end{proof}
Of course, we are not guaranteed that $X_1$ is a simplicial $T$-complex, 
as there may be many co-thin maps which do not factor through $X$. 
So, we simply repeat the process, and obtain $X_2$ with a map 
$X_1 \to X_2$, and a similar property. 
Continuing on gives us a sequence 
$$
X \to X_1 \to X_2 \to \cdots
$$
And we define $\stalg(X)$ to be the colimit of this diagram in $\text{AlgKan}$. 
Overloading notation, we denote each of the natural maps 
$X_k \to \stalg(X)$ by $\iota$. We then have the following:
\begin{lemma}
	A simplex $\alpha : \Delta^n \to \stalg(X)$ is a distinguished filler 
	iff there is some $n$ and some distinguished filler 
	$\tilde{\alpha}: \Delta^n \to X_k$ such that 
	$\alpha = \iota \circ \tilde{\alpha}$.
\end{lemma}
\begin{proof}
	This follows from the construction of filtered colimits given in \citep{Nikolaus2011}.
\end{proof}
\begin{lemma}
	A simplex $\alpha: \Delta^n \to \stalg(X)$ is thin iff 
	there is an $k \in \mathbb{N}$ and a thin simplex 
	$\tilde{\alpha}: \Delta_n \to X_k$ such that $\alpha = \iota \circ \tilde{\alpha}$.
\end{lemma}
\begin{proof}
	We proceed again by induction. 
	\begin{enumerate}
		\item Suppose $\alpha$ is degenerate. Then it is degenerate in some $X_k$, so this follows.
		\item Suppose that $\alpha$ is a distinguished filler. Then this follows from the previous lemma.
		\item Suppose that $\alpha$ is a composition of thin simplices.
			Then we can pull the horn which $\alpha$ is a composition of 
			back to some $X_n$ in which each face is thin, 
			and therefore the composition $\tilde{\alpha}$ is thin in $X_n$.
	\end{enumerate}
	\end{proof}
\begin{lemma}
	The mapping $X \to \stalg(X)$ defines a functor 
	$\text{AlgKan} \to \text{AlgKan}$ such that 
	$\stalg(X)$ has no distinct, co-thin simplices. 
	Furthermore, $\stalg(X)$ is initial among algebraic Kan complexes 
	equipped with a map from $X$ which have no distinct, co-thin simplices.
\end{lemma}
\begin{proof}
	Let $\alpha$ and $\beta$ be co-thin simplices in $\stalg(X)$. 
	Then by the previous lemma they factor through some finite stage 
	$X_n$ in which they are co-thin, and by construction are 
	therefore equal in $X_{n+1}$, and therefore $\stalg(X)$. 
	Functoriality follows from functoriality of the colimit 
	and Lemma \ref{lem:coequ-functorial}.
\end{proof}
\begin{theorem}
	The mapping $X \to \stalg(X)$ defines a functor $\algkan \to \stcom$ 
	(where the thin simplices of $\stalg(X)$ are as in 
	Definition \ref{dfn:thin-simplices}) left adjoint to the forgetful functor $\stcom \to \algkan$.
\end{theorem}
\begin{proof}
	Any map $X \to Y$ where $Y$ is a simplicial $T$-complex 
	factors uniquely through $\stalg(X)$, so this defines the 
	left adjoint to the forgetful functor 
	(as it is the inclusion of a full subcategory).
\end{proof}
We would like to take this adjunction and 
bump it up to the level of Quillen adjunction. 
To do this we will need to understand the 
model structure on the the category of simplicial T-complexes.
\subsection{Model Structure on Simplicial T-Complexes}
In \citep{Brown_Golasinski_89} it was shown that the category 
of crossed complexes admits a model structure with distinguished maps as follows:
\begin{itemize}
	\item The weak equivalences are those which induce 
		an isomorphism on $\pi_0$, $\pi_1$, and $H_n$ for $n \ge 2$.
	\item The fibrations are those maps $f: C \to D$ such that 
		$f$ given a $p \in C_0$ and $y \in D_n$ with 
		$\delta^n(y) = f(p)$, there exists a $z \in C_n$ with $f(z) = y$. 
	\item The cofibrations are all those maps which lift on the left of the acyclic fibrations.
\end{itemize}
From this, we easily obtain
\begin{theorem}
	The category of simplicial T-complexes admits a model structure where
	\begin{itemize}
		\item The weak equivalences are the weak equivalences of underlying simplicial sets.
		\item The fibrations are the Kan fibrations of underlying simplicial sets.
	\end{itemize}
\end{theorem}
\begin{proof}
	The model structure is the one transferred along 
	the equivalence of categories with the category of crossed complexes. 
	By Proposition 6.2 in \citep{Brown_Higgins_1991}, 
	the fibrations are the Kan fibrations on underlying simplicial sets. 
	By \ref{lem:weak-equiv-t-complex}, the weak equivalences are 
	the weak equivalences on underlying simplicial sets.
\end{proof}
\begin{corollary}
	The adjunction $\stalg \dashv \ualg$ is a Quillen adjunction. 
\end{corollary}  
We recall from earlier that the forgetful functor 
$U_A: \algkan \to \sset$ is the right adjoint of a Quillen adjunction.
\begin{definition}
	The functor $F_A: \sset \to \algkan$ is the left adjoint to 
	the forgetful functor $U_A: \algkan \to \sset$, 
	as defined in \citep{Nikolaus2011}.
\end{definition}
\begin{definition}
	We define $\st : \sset \to \stcom$ as the composition 
	$\stalg \circ F_A$ and $\ust: \stcom \to \sset$ as $U_A \circ \ualg$.
\end{definition}
\begin{corollary}
	The adjunction $\st \dashv \ust$ is a Quillen adjunction. 
\end{corollary}
\begin{proof}
	In \citep{Nikolaus2011} it is proven that $F_A \dashv U_A$ 
	is a Quillen adjunction. As a composition of Quillen adjunctions 
	is a Quillen adjunction, the result follows.
\end{proof}
The functor $\st$ therefore represents the universal 
``strictification'' of the $\infty$-groupoid represented 
by an arbitrary simplicial set. It is natural to ask 
how lossy this functor is. An easy general example is the following:
\begin{theorem}
	Let $X$ be an $n$-skeletal simplicial set. Then $\st(X)$ is $(n+1)$-coskeletal. 
\end{theorem}
\begin{proof}
	Note that by construction of $F_A$, 
	every simplex of dimension $n+1$ or higher is degenerate, 
	a distinguished filler, or a composition of simplices 
	which are either degenerate or distinguished fillers. 
	It follows then that in $\st(X)$ every simplex of 
	dimension $n+1$ or higher is thin. 
	Let $B: \partial \Delta^{N} \to X$ be the inclusion of a boundary, 
	where $N \ge n + 2$. Then note that a filler of any horn 
	must be a filler for the boundary, by axiom (3). 
	Of course, any filler for the boundary is also a 
	filler for each horn. Hence, by existence and uniqueness of horn fillers, 
	it must be the case that there is a unique extension of $B$ to $\Delta^N$.
\end{proof}
\begin{remark}
	This generalizes that if $X$ is $n$-skeletal, then $\mathbb{Z}[X]$ is $(n+1)$-coskeletal.
\end{remark}
\begin{corollary}
	Let $X$ be an $n$-skeletal simplicial set. 
	The map $\pi_\bullet(X) \to \pi_\bullet(\st(X))$ 
	contains in its kernel $\pi_{\ge n+1} (X)$.
\end{corollary}
This motivates the following definition:
\begin{definition}
	Let $X$ be a simplicial set. 
	The \emph{higher homotopy groups} of $X$ 
	are the graded pieces of the kernel of the map 
	$\pi_\bullet(F_A(X)) \to \pi_\bullet(\st(X))$.
\end{definition}
We call these groups the ``higher" homotopy groups 
because their dimension may exceed the geometric dimension of $X$; 
they are generated from lower dimensional data.
\begin{example}
	Let $X$ be $S^1 \lor S^2$. 
	Then the higher homotopy groups of $X$ are precisely $\pi_n(X)$ for $n \ge 3$. 
\end{example}

\begin{remark}
	A natural expectation would be that ``strictification'' 
	of $\infty$-groupoids should be idempotent - 
	after all, if $X$ is an $\omega$-groupoid, 
	sitting inside the category of weak $\infty$-groupoids (spaces), 
	the universal $\omega$-groupoid with a map from $X$ is just $X$ itself. 
	And indeed, the adjunction $\stalg \dashv \ualg$ is idempotent. 
	However, this adjunction is ``homotopically incorrect:'' 
	$\stalg$ does not preserve weak equivalences, 
	and so it is necessary to take the derived functors in order to 
	get the correct strictification adjunction. 
	The derived adjunction is \textit{not} idempotent: 
	for instance, $K(\mathbb{Z},2)$ ``is'' a strict $\infty$-groupoid 
	(that is, has a model as a simplicial $T$-complex). 
	However, the strictification is homotopy equivalent to 
	(the basepoint component of) $\mathbb{Z}[K(\mathbb{Z},2)]$, 
	by Theorem \ref{thm:simply-connected-crossed-complex}.
\end{remark}

\section{The induced coalgebra $\st(X)$ detects the homotopy type of $X$}
The adjunction $\st \dashv \ust$ induces a free functor 
$\freecoalg: \sset \to \text{CoAlg}(\st)$, 
where $\text{CoAlg}(\st)$ is the category of coalgebras 
for the comonad $\st \circ \ust$ induced by the adjunction $\st \dashv \ust$.  
The free functor $\freecoalg$ is given on objects by 
$$
\freecoalg(X) = \left( \st(X), \st(\eta_X) : \st(X) \to (\st \circ \ust) (\st (X)) \right)
$$
where $\eta$ is the unit of the adjunction $\st \dashv \ust$. 
We will show in this section that $\freecoalg(X)$ 
determines the homotopy type of $X$. 
In the following section we will upgrade to an equivalence of quasicategories. 
The precise theorem in this section is the following:
\begin{theorem}\label{thm:main-theorem-basic}
	For $X \in \sset$, denote by $\st^\bullet(X)$ 
	the cosimplicial object of $\sset$ whose $n$th space is $(\ust \circ \st)^{n+1}(X)$ 
	and whose coface and codegeneracy maps are induced by the coalgebra structure of $\st(X)$. 
	The natural map
	$$
		X \to \holim \, (\st^\bullet(X))
	$$
	is a weak equivalence.  
\end{theorem}
Bousfield and Kan proved a similar result for nilpotent spaces 
and simplicial abelian groups in \citep{Bousfield_Kan_1987},
which we rephrase as follows:
\begin{theorem}[\citep{Bousfield_Kan_1987}, III.5.4]\label{thm:BK-main-theorem}
	Let $X \in \sset$ be nilpotent. 
	Denote by $\tz^\bullet[X]$ the cosimplicial simplicial set 
	whose $n$th space is $(\tz[-])^n(X)$, with coface and codegeneracy maps 
	given by the coalgebra structure on $\tz[X]$. 
	Then there is a natural weak equivalence 
	$$
		X \to \holim \, (\tz^\bullet[X])
	$$
\end{theorem}
Recall that $\mathbb{Z}[-]$ is the functor which 
takes a simplicial set and applies the free abelian group functor 
at each level to produce a simplicial abelian group, 
and $\tilde{\mathbb{Z}}[-]$ is the functor which takes a \emph{pointed} 
simplicial set and applies $\mathbb{Z}[-]$, 
and then quotients out by the subgroup generate by the 
sub-simplicial set consisting of the basepoint (and all its degeneracies).
Our proof will essentially be an extension of 
the Bousfield-Kan result to the general case, 
using the identification between simply connected 
simplicial $T$-complexes and simply connected 
chain complexes that we established in section 3.

\begin{remark}
	While the main result and the majority of the work 
	in this paper is done unpointed, 
	we will need to use pointed spaces in order to relate 
	$\st(X)$ to $\tz[X]$ and thereby 
	leverage the Bousfield-Kan completion. 
	The adjunctions $\st \dashv \ust$ and 
	$\nonabdoldnormalizer \dashv \nonabdoldsset$ naturally extend to pointed versions.
\end{remark}

\begin{lemma}\label{lem:sc-strictification-abelianization}
	Let $X$ be a simply connected pointed simplicial set. 
	There is a natural weak equivalence 
	$$\ust(\st(X)) \to |\tz[X]|$$
\end{lemma}
\begin{proof}
	Let $A \to X$ be a weak equivalence, 
	with $A$ a $1$-reduced pointed simplicial set. 
	We then have the commutative square
	\begin{center}
	\begin{tikzcd}
		\nonabdoldsset(\st(A)) 
			\ar[d] 
			\ar[r] 
		& \ucomplex(\tabcomplex(\nonabdoldsset(\st(A)))) 
			\ar[d] \\
		\nonabdoldsset(\st(X)) 
			\ar[r] 
		& \ucomplex(\tabcomplex(\nonabdoldsset(\st(X)))) 
	\end{tikzcd}
	\end{center}	
Where $\tabcomplex$ being the reduced version of $\abcomplex$. 
The downwards arrows are weak equivalences by basic model category theory, 
using that all objects of $\ch$ and $\crcom$ are fibrant 
and all objects of $\sset$ are cofibrant. 
The top arrow is an isomorphism by Theorem \ref{thm:simply-connected-crossed-complex}. 
Therefore, the bottom horizontal arrow is a weak equivalence. 
Finally, the equivalence between the simplicial and the chain complex models 
(item (4) of Theorem \ref{thm:simply-connected-crossed-complex}) completes the argument.
\end{proof}

\begin{lemma}\label{lem:homology-strictification-isomorphism}
	Let $X \in \sset$ be connected. 
	For $n \ge 2$ and for any choice of basepoint there is a natural isomorphism of functors
	$$
		{H}_n(\widehat{X}) \cong \pi_n(\ust(\st(X)))	
	$$
	Where $\widehat{X}$ is the universal cover of $X$.
	For $n=0,1$ and $x \in X$ there is a natural isomorphism
	$$
		\pi_n(X) \cong \pi_n(\ust(\st(X)))	
	$$
\end{lemma}
\begin{proof}
	The $n=0$ case follows readily from the construction, 
	as adding horn fillers and quotienting co-thin simplices do not change $\pi_0$. 
	The $n=1$ case follows from \ref{lem:weak-equiv-t-complex}, 
	and that $H_1(\nonabdoldsset(\st(X)), x)$ is by definition 
	$\pi_1(X_1,x)$ modulo the boundaries of $2$-simplices. 
	For the $n \ge 2$ case, in (\citep{Brown_Higgins_Sivera_2011}, Section 8.4) 
	it is proven that $H_n(\widehat{X}) \cong H_n(\nonabdoldsset(\st(X)))$. 
	The identification then follows from Lemma \ref{lem:weak-equiv-t-complex}.
\end{proof}
\begin{lemma}\label{lem:strictification-1-connected}
	The map $X \to \ust(\st(X))$ is $1$-connected.
\end{lemma}
\begin{proof}
	Note that that we have a factorization $X \to \ust(\st(X)) \to |\mathbb{Z}[X]|$, 
	which when we apply $\pi_1$ gives us $\pi_1(X) \to \pi_1(X) \to \pi_1(X)^{\text{ab}}$ 
	by \ref{lem:homology-strictification-isomorphism}, 
	where the composite is the abelianization.
	Let $\alpha \in \pi_1(X)$. 
	Considering a map $S^1 \to X$ representing $\alpha$ 
	and applying functoriality of the sequence, 
	we obtain the commutative diagram
	\begin{center}
	\begin{tikzcd}
		\mathbb{Z} 
			\ar[r] 
			\ar[d] 
		& \mathbb{Z} 
			\ar[d] \\
		\pi_1(X) 
			\ar[r] 
		& \pi_1(\ust(\st(X))) 
	\end{tikzcd}
	\end{center}
	Where the top arrow is an isomorphism, 
	as if the composition $\mathbb{Z} \to \mathbb{Z} \to \mathbb{Z}$ 
	is an isomorphism, each map must be. 
	We then see that $\pi_1(X) \to \pi_1(\ust(\st(X))$ 
	must be the identity under the natural isomorphism 
	$\pi_1(X) \cong \pi_1(\ust(\st(X)))$.
\end{proof}
\begin{lemma}\label{lem:strictification-preserves-groupoids}
	If $G$ is a $1$-type (that is, if $\pi_i(G) \cong 0$ for $i > 1$), 
	then the unit $G \to \ust(\st(G))$ is a weak equivalence.
\end{lemma}
\begin{proof}
	This follows from Lemmas \ref{lem:strictification-1-connected} 
	and \ref{lem:homology-strictification-isomorphism}.
\end{proof}
\begin{definition}
	By a \emph{homotopy fibre sequence} we shall mean a sequence 
	of pointed simplicial sets $X \to Y \to Z$, such that 
	\begin{center}
	\begin{tikzcd}
			X 
				\ar[r] 
					\ar[d] 
			& * 
				\ar[d] \\
			Y 
				\ar[r] 
			& Z
	\end{tikzcd}
	\end{center}
	is a homotopy pullback square. 
\end{definition}
For our purposes, it will really be sufficient just to know 
that the property of being a homotopy fibre sequence 
is invariant under weak equivalence of diagrams, 
and that for a pointed, connected simplicial set $X$ the sequence
$$\widehat{X} \to X \to X_{\le 1}$$
where again $\widehat{X}$ is the universal cover, 
and $X_{\le 1}$ is the $2$-coskeleton of $X$, 
is a homotopy fibre sequence. 
\begin{lemma}\label{lem:strictication-preserves-covers}
	If $Y \to X \to G$ is a fibre sequence of pointed, 
	connected simplicial sets where $G$ is a $1$-type
	and $X \to G$ induces an isomorphism on $\pi_1$, then 
	$$
	\ust(\st(Y)) \to \ust(\st(X)) \to \ust(\st(G))
	$$
	is a fibre sequence.
\end{lemma}
\begin{proof}
	The given hypotheses tell us that $Y \to X \to G$ 
	is weak homotopy equivalent to the standard homotopy fibre sequence 
	$\widehat{X} \to X \to X_{\le 1}$. 
	As $\ust$ and $\st$ preserve weak equivalences, 
	it suffices to show therefore that 
	$$ \ust(\st(\widehat{X})) \to \ust(\st(X)) \to \ust(\st(X_{\le 1})) $$
	is a homotopy fibre sequence, 
	which we will do by showing it is equivalent to the standard homotopy fibre sequence
	$$ \widehat{\ust(\st(X))} \to \ust(\st(X)) \to \ust(\st(X))_{\le 1} $$
	By \ref{lem:strictification-1-connected} and \ref{lem:strictification-preserves-groupoids}, 
	the map $\ust(\st(X)) \to \ust(\st(X_{\le 1}))$ is $1$-connected 
	and $\ust(\st(X_{\le 1}))$ is a $1$-type. 
	Hence, it is weak equivalent to $\ust(\st(X))_{\le 1}$. 
	Further, we can lift our map 
	$\ust(\st(\widehat{X})) \to \ust(\st(X))$ to get the dashed arrow in
	\begin{center}
	\begin{tikzcd}
		& \widehat{\ust(\st(X))} \ar[d] \\
		\ust(\st(\widehat{X})) \ar[ru, dashed] \ar[r] & \ust(\st(X))
	\end{tikzcd}
	\end{center}
	Invoking Lemma \ref{lem:homology-strictification-isomorphism}, 
	we obtain that this map is an isomorphism on 
	$\pi_n$ for $n > 2$ and therefore is a weak equivalence.
	Putting this together, we have a weak equivalence of diagrams
	\begin{center}
	\begin{tikzcd}
		\ust(\st(\widehat{X})) 
			\ar[r] 
			\ar[d] 
		& \ust(\st(X)) 
			\ar[r] 
			\ar[d] 
		& \ust(\st(X_{\le 1})) 
			\ar[d] \\
		\widehat{\ust(\st(X))} 
			\ar[r] 
		& \ust(\st(X)) 
			\ar[r] 
		& \ust(\st(X))_{\le 1}
	\end{tikzcd}
	\end{center}
	In which the left vertical map and the right vertical map 
	are the two weak equivalences we just described, 
	and the middle vertical map is the identity. 
	We therefore have a weak equivalence of diagrams as desired.
\end{proof}
\begin{proof}[Proof of Theorem \ref{thm:main-theorem-basic}]
	First consider the case where $X$ is connected, 
	and pick a basepoint for $X$. 
	We have a fibre sequence $\widehat{X} \to X \to X_{\le 1}$, 
	where $X_{\le 1}$ is the $2$-coskeleton of $X$ and $\widehat{X}$ 
	is therefore (up to homotopy) a universal cover of $X$. 
		We consider the following diagram:
	\begin{center}
	\begin{tikzcd}
		\widehat{X} 
			\ar[d] 
			\ar[r] 
		& \holim(\st^\bullet(\widehat{X})) 
			\ar[d] 
			\ar[r] 
		& (\ust \circ \st)(\widehat{X}) 
			\ar[d] 
			\ar[r, shift left = 2] 
			\ar[r, shift right = 2] 
		& (\ust \circ \st)^2(\widehat{X}) 
			\ar[d]  
			\ar[r, shift left = 2] 
			\ar[r] 
			\ar[r, shift right = 2] 
		& \cdots \\
		X 
			\ar[d] 
			\ar[r] 
		& \holim(\st^\bullet(X)) 
			\ar[d] 
			\ar[r] 
		& (\ust \circ \st)(X) 
			\ar[d] 
			\ar[r, shift left = 2] 
			\ar[r, shift right = 2] 
		& (\ust \circ \st)^2(X) 
			\ar[d] 
			\ar[r, shift left = 2] 
			\ar[r] 
			\ar[r, shift right = 2] 
		& \cdots \\
		X_{\le 1} 
			\ar[r] 
		& \holim(\st^\bullet(X_{\le 1})) 
			\ar[r] 
		& (\ust \circ \st)(X_{\le 1}) 
			\ar[r, shift left = 2] 
			\ar[r, shift right = 2] 
		& (\ust \circ \st)^2(X_{\le 1}) 
			\ar[r, shift left = 2] 
			\ar[r] 
			\ar[r, shift right = 2]  
		& \cdots 
	\end{tikzcd}	
	\end{center}
	By Lemma \ref{lem:strictication-preserves-covers}, each of the sequences
	\begin{center}
	\begin{tikzcd}
	 	(\ust \circ \st)^n(\widehat{X}) 
	 		\ar[r] 
	 	& (\ust \circ \st)^n(X) 
	 		\ar[r] 
	 	& (\ust \circ \st)^n(X_{\le 1})  
	\end{tikzcd}
	\end{center}
	is a fibre sequence. Since homotopy limits commute, it follows that
	\begin{center}
	\begin{tikzcd}
		\holim(\st^\bullet(\widehat{X})) 
			\ar[r] 
		& \holim(\st^\bullet(X)) 
			\ar[r] 
		& \holim(\st^\bullet(X_{\le 1}) )
	\end{tikzcd}
	\end{center}
		is a fibre sequence. 
		By Lemma \ref{lem:sc-strictification-abelianization}, 
		we have a commutative diagram 
	\begin{center}
	\begin{tikzcd}
		\widehat{X} 
			\ar[rd] 
			\ar[r] 
		& \holim(\st^\bullet(\widehat{X})) 
			\ar[d] \\
	 & \holim(\tz^\bullet(\widehat{X})) 
	\end{tikzcd}
	\end{center}
	Where the downwards arrows is a weak equivalence. 
	By \ref{thm:BK-main-theorem}, 
	$\widehat{X} \to \holim(\tz^\bullet(\widehat{X}))$ is a weak equivalence, 
	so $\widehat{X} \to \holim((\ust \circ \st)^\bullet(\widehat{X}))$ is as well. 
	By Lemma \ref{lem:strictification-preserves-groupoids}, 
	we get a commutative diagram
	\begin{center}
	\begin{tikzcd}
		X_{\le 1} 
			\ar[r]  
			\ar[rd] 
		& \holim(\st^\bullet(X_{\le 1}))  
			\ar[d] \\
	 	& X_{\le 1}
	\end{tikzcd}
	\end{center}
	Where the downwards arrow is a weak equivalence. 
	Therefore, the map $X_{\le 1} \to \holim(\st^\bullet(X_{\le 1}))$ 
	is a weak equivalence. It follows by the five lemma that 
	$X \to \holim(\st^\bullet(X))$ is a weak equivalence.
	\\
	\indent Finally, for the general case 
	(where $X$ may not be connected): 
	Both the functors $\ust$ and $\st$ commute with coproducts, 
	and therefore the map $X \to \holim(\st^\bullet(X))$  
	is the coproduct of the map on each of the connected components, 
	from which the result follows. 
\end{proof}
\section{The adjunction $\st \dashv \ust$ induces a comonadic adjunction of $\infty$-categories}
In this section we will do a little more work 
to achieve our main goal: providing an equivalence 
of homotopy theories between spaces and coalgebras in $\stcom$ 
for the comonad $\st \circ \ust$. 
We will do this via an application of the Barr--Beck--Lurie theorem. 
The main theorem of this section is
\begin{theorem}\label{thm:main}
	Let $\lf: \qsset \rightleftarrows \qstcom : \rf$ 
	be the adjunction of quasicategories induced by the Quillen adjunction 
	$\st: \sset \rightleftarrows \stcom :\ust$. 
	Then $\lf \dashv \rf$ is comonadic: that is, 
	it induces an equivalence of quasicategories between the quasicategory 
	$\qsset$ and the quasicategory of coalgebras for the comonad $\lf \circ \rf$.
\end{theorem}
\subsection{The induced functor of quasicategories}
In general, a Quillen adjunction of model categories induces 
an adjunction of quasicategories. Here, our categories are rather nice: 
they are both simplicially enriched. 
It is therefore reasonable to expect that we might be able to use this 
simplicial enrichment to define the $\infty$-categorical adjunction we seek. 
Unfortunately, the functor $\ust$ is not simplicial (see \citep{Tonks_2003}), 
which complicates using these tools. 
Therefore, we ignore entirely the underlying simplicial enrichment 
and use basic model categorical tools. 
\\
\indent In (\citep{Mazel-Gee_2015}, A.3), a general  technique 
for obtaining adjunctions of quasicategories from model categories 
with functorial (co)fibrant replacement is described. 
Fortunately, we are in the simplest possible case: 
$\sset$ has as its cofibrant replacement functor the identity functor, 
and $\stcom$ has as its fibrant replacement functor the identity functor. 
Hence, the argument there allows us to conclude the following: 
\begin{theorem}
	The adjunction $\st \dashv \ust$ induces an adjunction 
	of quasicategories $\lf: \qsset \rightleftarrows \qstcom : \rf$ 
	with unit transformation induced directly by the 
	unit transformation $\mathbf{1}_{\sset} \to \ust \circ \st$. 
\end{theorem}
To see this argument more fleshed out, see the appendix.
\subsection{Comonadicity}
In order to prove our main theorem, 
we will apply the Barr-Beck-Lurie Theorem:
\begin{theorem}[\citep{Lurie_2009}, Theorem 4.7.2.2]
	A functor $F: \C \to \D$ exhibits $\C$ as 
	comonadic over $\D$ if and only if it admits a right
	adjoint, is conservative, and preserves all limits 
	of $F$-split coaugmented cosimplicial objects.
\end{theorem}
Our aim is to prove that the functor $\lf: \qsset \to \qstcom$ 
exhibits $\qsset$ as comonadic over $\qstcom$. 
Of course, we already have a right adjoint. 
Conservativity of $\lf$ is, as it turns out, a classical fact:
\begin{lemma}\label{lem:strictification-conservative}
	$\lf : \qsset \to \qstcom$ is conservative.
\end{lemma}
\begin{proof}
	We want to prove that for a morphism $f : X \to Y$ in $\qsset$, 
	if $\lf(f)$ is a weak equivalence, then so is $f$. 
	If $\lf(f)$ is a weak equivalence, then it induces isomorphisms 
	on all the homotopy groups of the underlying simplicial sets of 
	$\lf(X)$ and $\lf(Y)$. Equivalently by Lemma 
	\ref{lem:homology-strictification-isomorphism}, it induces an isomorphism 
	on all homology groups of the associated crossed complexes. 
	By \ref{lem:homology-strictification-isomorphism}, this is equivalent to 
	saying that $f$ induces an isomorphisms on $\pi_0(X) \to \pi_0(Y)$, 
	$\pi_1(X) \to \pi_1(Y)$ for any choice of basepoint, and 
	$H_n(\widehat{X}) \to H_n(\widehat{Y})$ for $n \ge 2$. 
	By (\citep{Dieck_2010}, 20.1.8), $f$ is a weak homotopy equivalence. 
\end{proof}
It therefore remains to verify the last condition about 
$F$-split coaugmented cosimplicial objects.
The following lemma simplifies the work to be done:
\begin{lemma}[\citep{Holmberg-Peroux_2020}, Prop 6.1.4]
	Given a pair of adjoint functors $L: \C \rightleftarrows \D : R$ in quasicategories, 
	such that $L$ is conservative. 
	Then $L$ is comonadic if and only if the map 
	$X \to \lim^\C_\Delta (RL^{\bullet +1}X)$ 
	is an equivalence for all objects $X$ in $\C$.
\end{lemma}
\begin{proof}[Proof of \ref{thm:main}]
	By Lemma \ref{lem:strictification-conservative} and the above lemma, 
	we only need to prove that for any $X \in \qsset$, 
	the map $X \to \lim^\C_\Delta (\rf\lf^{\bullet +1}X)$ is an equivalence. 
	However, by our construction of the quasicategorical adjunction, 
	the cosimplicial object $\rf\lf^{\bullet +1}X$ is the same as 
	the diagram $\st^\bullet X$ as in Theorem \ref{thm:main-theorem-basic}. 
	Since homotopy limits compute limits in the associated quasicategory, 
	this then follows directly by Theorem \ref{thm:main-theorem-basic}.
\end{proof}
\begin{remark}
	We have phrased our main theorem as being about 
	simplicial sets and simplicial $T$-complexes, 
	but of course from the equivalence of categories 
	$\stcom \rightleftarrows \stinfty$ the analagous 
	statement for $\omega$-groupoids immediately follows.
\end{remark}
\section{Future Work}
\begin{itemize}
	\item The most immediate line of future inquiry 
		is to extend this work to categories, not just groupoids: 
		that is, construct adjunctions between categories of 
		strict and weak $(\infty,n)$ categories, and determine if 
		these are comonadic. Perhaps the most full extension along 
		these lines would be to construct for any $m \ge n$ a 
		strictification functor from weak $(m,n)$-categories to 
		strict $(m,n)$-categories, and of course show that the 
		resulting diagram of homotopy theories indexed by 
		$\mathbb{N}^2 \times [1]$ commutes. 
	\item One consequence of this work is a globular model for homotopy types: 
		the objects are strict $\infty$-groupoids equipped with a coalgebra map 
		for the comonad $\ust \nonabdoldsset \circ \nonabdoldnormalizer \circ \st$. 
		However,this description strongly relies on simplicial sets as an 
		intermediary to describe the comonad. Thus, one goal is to construct a 
		comonad of $\stinfty$ which is equivalent to the comonad induced by 
		the comonad on $\stcom$, but which is more ``inherently globular.'' 
	\item Although this work proves that $\omega$-groupoids 
		can be used to model weak $\infty$-groupoids, 
		it does not provide a convenient model category of coalgebras.
		A strengthening of this result would be to construct a model 
		structure on the category of coalgebras for the comonad, 
		along with a Quillen equivalence to the model category of simplicial sets.
\end{itemize}
\appendix 
\section{The passage to quasicategories}
In this section we flesh out in greater detail the passage 
from the Quillen adjunction $\st : \sset \rightleftarrows \stcom : \ust$ 
to the adjunction of quasicategories $\lf: \qsset \rightleftarrows \qstcom : \rf$, 
in order to fully justify our results. 
There is nothing original in this section: 
we simply flesh out some of the argument in \citep{Mazel-Gee_2015} 
as it applies to our particular case of interest, 
in order to make it more accessible to the non-expert reader. 
\\\\
\indent In order to obtain a quasicategory from a relative category 
(in particular, a model category), we shall pass through 
complete Segal spaces, which are the fibrant objects 
in a certain model structure on bisimplicial sets. 
We will not exposit much of the theory of complete Segal spaces, 
but direct the interested reader to \citep{Rezk_1998}. 
We begin with our model category, viewed as a relative category $(M,\mathcal{W})$. 
We then apply the following two functors:
\begin{center}
\begin{tikzcd}
	\text{RelCat} 
		\ar[r, "N_r"] 
	& \text{ssSet} 
		\ar[r, " i^*_1"] 
	& \text{sSet}
\end{tikzcd}
\end{center}
The first functor $N_r$ is called the ``Rezk nerve'' 
(originally called the ``classifying diagram'' in \citep{Rezk_1998}). 
It is given as follows:
\begin{definition}
	Given a relative category $(M, \mathcal{W})$, 
	the Rezk nerve $N_r(M)$ is given by 
	$(n,m) \mapsto \text{Fun}([n], \text{Core}(\text{Fun}([m], M)))$.
\end{definition}
Here ``Core'' denotes the subcategory $\mathcal{W} \subset M$. 
Importantly, $N_r$ preserves products, and so a natural transformation 
$M \times I \to M$ is sent to $N_r(M) \times N_r(I) \to N_r(M)$. 
Here $I$ is the standard interval category, 
equipped with the relative category structure where only the isomorphisms 
(the two identity arrows) are weak eqivalences. 
We wish to obtain a quasicategory from this. 
First, we fibrantly replace $N_r(M)$ in the Reedy model structure 
to get a bisimplicial set denoted $N_R^f(M)$, which is a complete Segal space.
We note that $N_R(I)$ is already a complete Segal space, 
as $N_r(C)$ is a complete Segal space whenever $C$ 
is equipped with the minimal relative category structure. 
Applying this all to our starting situation of the unit 
natural transformation of the adjunction $\st \dashv \ust$, 
we have a map
$$N_r^f(\stcom) \times N_r(I) \to N_r(\stcom)$$
whose component at any object of $N_r^f(\stcom)$ 
is precisely the unit map $X \to \ust(\st(X))$. 
We now apply the functor $i^*_1$ of \citep{Joyal_2006} 
to obtain a quasicategory. Conveniently, this functor merely 
restricts a bisimplicial set $X_{\star \star}$ to its 
first row $X_{\star 0}$, and so we have now a diagram
$$i^*_1\left( N_r^f(\sset)\right) \times N(I) \to i^*_1\left( N_r^f(\sset)\right)$$
Where $N(I)$ is the ordinary nerve of the category $I$. 
We take $i^*_1\left( N_r^f(\sset)\right)$ as our model 
for the quasicategory of simplicial sets $\qsset$ 
(and similarly $i^*_1\left( N_r^f(\stcom)\right)$ 
models the quasicategory of simplicial $T$-complexes $\qstcom$). 
This map is a candidate unit map for the quasicategorical adjunction 
$\lf : \qsset \rightleftarrows \qstcom : \rf$ in the sense 
of (\citep{Lurie_2009}, Definition 5.2.2.7); 
to show that it truly provides a quasicategorical adjunction 
it must be shown that it induces weak equivalences
$$
\underline{Map}_{\qstcom}(\lf(x),y) 
	\to \underline{Map}_{\qsset}(\rf (\lf (x)), \rf(y)) 
	\to \underline{Map}_{\qsset}(x, \rf(y))
$$
For any $x$ and $y$ objects in $\qsset$, 
where $\underline{Map}_\mathcal{C}$ is the mapping space 
between objects in a quasicategory. 
This follows from the fact that these mapping spaces are 
functorially equivalent to the mapping spaces given by hammock localization, 
which in turn are equivalent by \citep{Dwyer_Kan_1980} 
to mapping spaces calculated via cosimplicial resolutions. 
This reduces the problem to checking the weak equivalence on the level of model categories, 
where we can apply (\citep{Dwyer_Kan_1980}, Proposition 5.4) 
to conclude that the induced map of mapping spaces is an isomorphism. 
\\
\indent We now have that the unit transformation for the adjunction of 
quasicategories can be taken to be the unit transformation at the level of categories 
(more precisely, that the component at any object $X$ is the arrow $X \to (\ust \circ \st)(X)$). 
Hence the canonical cosimplicial resolution induced by the 
adjunction of quasicategories is equivalent to the canonical 
cosimplicial resolution induced by the adjunction of model categories, 
and therefore addressing the question of whether it is a 
limit diagram in the quasicategory $\qsset$ reduces to the 
question of whether it is a homotopy limit diagram in the 
simplicial model category $\sset$, justifying our usage of 
Theorem \ref{thm:main-theorem-basic} to prove Theorem \ref{thm:main}.

\bibliographystyle{plainnat}

\end{document}